\documentclass[leqno,10pt]{amsart}
\usepackage{amsfonts, amsmath, amssymb, amsthm, color}
\allowdisplaybreaks[1]
\newtheorem{thm}{Theorem}[section]
\newtheorem{lemma}{Lemma}[section]
\newtheorem{prop}{Proposition}[section]
\newtheorem{cor}{Corollary}[section]

\theoremstyle{definition}
\numberwithin{equation}{section}
\newcommand{\rr}{\mathbb R}
\newcommand{\al}{\alpha}
\newcommand{\de}{\delta}
\newcommand{\eps}{\epsilon}
\newcommand{\la}{\lambda}

\newcommand{\ee}{\varepsilon}
\newcommand{\bra}{\langle}
\newcommand{\ket}{\rangle}
\newcommand{\Om}{\Omega}

\newcommand{\calI}{\mathcal I}

\newcommand{\pl}{\partial}

\newcommand{\calP}{\mathcal P}
\newcommand{\calL}{\mathcal L}
\newcommand{\calJ}{\mathcal J_\lambda}
\newcommand{\calF}{\mathcal F}

\newcommand{\IO}{\Omega\times I}

\newcommand{\rhoa}{\rho_\alpha}

\newcommand{\dtx}{d\tilde x}
\newcommand{\tO}{\widetilde\Omega}
\newcommand{\tx}{\tilde x}
\newcommand{\tG}{\widetilde{G}}
\newcommand{\tA}{\widetilde\Gamma_\lambda}
\newcommand{\tilA}{\widetilde A}
\newcommand{\tilE}{\widetilde E}

\newcommand{\km}{k^M}
\newcommand{\Ue}{U_\eps}
\newcommand{\psie}{\psi_\eps}
\newcommand{\rhoe}{\rho_\eps}

\def\bge{\begin{eqnarray}}
\def\bgee{\begin{eqnarray*}}
\def\ege{\end{eqnarray}}
\def\egee{\end{eqnarray*}}
\DeclareMathOperator{\dive}{div}
\DeclareMathOperator{\supp}{supp}
\def\sideremark#1{\ifvmode\leavevmode\fi\vadjust{\vbox to0pt{\vss% the remark
 \hbox to 0pt{\hskip\hsize\hskip1em%                          will appear only
 \vbox{\hsize3cm\tiny\raggedright\pretolerance10000%          on the side
  \noindent #1\hfill}\hss}\vbox to8pt{\vfil}\vss}}}%

\begin{document}
\title[A multi-species chemotaxis system]{A multi-species chemotaxis system:
Lyapunov functionals, duality, critical mass}
\author{N.I.~Kavallaris}
\address[N.I.~Kavallaris]{Department of Mathematics,
Faculty of Science and Engineering, University of Chester,
Thornton Science Park, Chester CH2 4NU, U.K.}
\email{n.kavallaris@chester.ac.uk}
\author{T.~Ricciardi${}^\ast$}\thanks{${}^\ast$Corresponding author}
\address[T.~Ricciardi] {Dipartimento di Matematica e Applicazioni,
Universit\`{a} di Napoli Federico II, Via Cintia, Monte S.~Angelo, 80126 Napoli, Italy}
\email{tonricci@unina.it}
\author{G.~Zecca}
\address[G.~Zecca] {Dipartimento di Matematica e Applicazioni,
Universit\`{a} di Napoli Federico II, Via Cintia, Monte S.~Angelo, 80126 Napoli, Italy}
\email{g.zecca@unina.it}
%\date{\today}
\begin{abstract}
We introduce a multi-species chemotaxis type system admitting an
arbitrarily large number of population species,
all of which are attracted vs.\ repelled
by a single chemical substance. The production vs.\ destruction rates of the
chemotactic substance by the species is described by a probability measure.
For such a model we investigate the variational structures,
in particular we prove the existence of Lyapunov functionals,
we establish duality properties as well as a logarithmic Hardy-Littlewood-Sobolev type inequality
for the associated free energy.
The latter inequality provides the optimal critical value for the conserved total population mass.
\end{abstract}
\keywords{Multi-species chemotaxis models, Lyapunov functionals, duality, logarithmic Hardy-Littlewood-Sobolev inequality,
Moser-Trudinger inequality}
\subjclass[2000]{35K51, 35J20, 26D15, 92C17}
\maketitle
%%%%%%%%%%%%%%%%%%%%%%%%%%%%%%%%%%%%%%%%%%%%%%%%%%%%%%%%%%%%%%%%%%%%%%%%%%%%%%%
%%%%%%%%%%%%%%%%%%%%%%%%%%%%%%%%%%%%%%%%%%%%%%%%%%%%%%%%%%%%%%%%%%%%%%%%%%%%%%%~~~~~
%%%%%%%%%%%%%%%%%%%%%%%%%%%%%%%%%%%%%%%%%%%%%%%%%%%%%%%%%%%%%%%%%%%%%%%%%%%%%%%
%%%%%%%%%%%%%%%%%%%%%%%%%%%%%%%%%%%%%%%%%%%%%%%%%%%%%%%%%%%%%%%%%%%%%%%%%%%%%%%
\section{Introduction and motivation}
\label{sec:intro}
%%%%%%%%%%%%%%%%%%%%%%%%%%%%%%%%%%%%%%%%%%%%%%%%%%%%%%%%%%%%%%%%%%%%%%%%%%%%%%%
%%%%%%%%%%%%%%%%%%%%%%%%%%%%%%%%%%%%%%%%%%%%%%%%%%%%%%%%%%%%%%%%%%%%%%%%%%%%%%%
Since the pioneering chemotaxis model of Keller and Segel \cite{KellerSegel},
see also Patlak \cite{Patlak1953},
several models have been introduced in order to describe the chemotactic movement of motile species,
such as the slime mold \textit{Dictyostelium Discoideum}.
In particular, much attention has been devoted in recent years to derive multi-species
chemotactic models, see  \cite{ESV2009, Ho2011, Wo2002, Wo2014} and the references therein.
\par
Our aim in this note is to introduce and to analyze, particularly from the variational point of view,
a new multi-species parabolic-parabolic chemotaxis system involving an
arbitrarily large number of population species $\rhoa$, depending on the (possibly continuous) index
$\al\in[-1,1]$, and a single chemical $v$.
Such  a \lq\lq continuous index" will turn out to be useful in order to efficiently
formulate, in terms of a probability distribution $\calP(d\al)$
defined on the index range $[-1,1]$, the variational structures of the system,
as well as to describe relevant quantities such as the conserved total population mass
and the overall chemical production rate.
We assume that $\rhoa$ and $v$ are defined on a two-dimensional domain,
which is a natural setting for species raised in a cell-culture dish.
In our model, some of the population species are attracted by the substance $v$, while others are repelled by
it, with different (normalized) intensities given by the value $\al\in[-1,1]$,
where positive values of $\al$ correspond to attraction whilst negative values correspond to repulsion.
In turn, the substance is self-produced by those species it attracts, and destroyed by
those species it repels. In particular, this model fits the \lq\lq absence of conflicts"
definition introduced in \cite{Wo2002}. Birth and death rates are neglected.
\par
We are particularly interested in the limit case where the dynamics of the population species is significantly faster than the
dynamics of the chemical. In this case, our system may be written as an evolution problem
for the chemical substance $v$ only. We further assume that the total mass of \textit{all} the population species,
is conserved in time. This assumption is natural when the different species are produced by a cell differentiation process as occurs, e.g.,
in the early aggregation stages of the Dictyostelium during mound formation
\cite{ESV2009, Weijer}.
\par
More precisely, we consider the following system:
\begin{equation}
\label{eq:detsys}
\left\{
\begin{aligned}
&\de_\al\frac{\pl\rhoa}{\pl t}=\Delta\rhoa-\al\dive(\rhoa\nabla v),&&\hbox{in\ }\Omega\times(0,T),\ \al\in[-1,1]\\
&\ee\frac{\pl v}{\pl t}=\Delta v+\int_{[-1,1]}\al\rhoa\,\calP(d\al),&&\hbox{in\ }\Omega\times(0,T)\\
&\nu\cdot(\nabla\rhoa-\al\rhoa\nabla v)=0,\ v=0, &&\hbox{on\ }\pl\Omega\times(0,T)\\
&\rhoa(x,0)=\rhoa^0(x)\ge0,\quad v(x,0)=v^0(x),&&\hbox{in\ }\Omega,\\
\end{aligned}
\right.
\end{equation}
where $\Om\subset\rr^2$ is a smooth bounded domain, $\nu$ denotes the outer unit normal vector on $\pl\Omega$,
$T>0$ stands for the maximum existence time for \eqref{eq:detsys},
$\al\in[-1,1]$, $\calP\in\mathcal M([-1,1])$ is a probability measure,
$v^0\in H_0^1(\Om)$ and the constants $\ee,\de_\al$ satisfy $\ee>0$, $\de_\al\ge\de_0$
for some $\de_0>0$.
We observe that if $\supp\calP\subset[0,1]$, namely if $\calP$ is positively supported
(see \eqref{def:suppP} below for the precise definition of $\supp\calP$), then $v^0\ge0$ implies $v\ge0$ by the maximum principle.
On the other hand, if $\supp\calP\cap[-1,0)\neq\emptyset$, the function $v$ is not necessarily non-negative.
In this case, $v$ is interpreted as \lq\lq chemical potential", see \cite{Ho2011}.
\par
The evolution equation for $\rhoa$, together with the no-flux boundary condition in system~\eqref{eq:detsys},
implies the conservation in time of the population mass,
for each population $\rhoa$ separately:
\begin{equation}
\label{eq:massintro}
\int_\Om\rhoa(x,t)\,dx=\int_\Om\rhoa^0(x)\,dx \qquad\mbox{for all }\al\in[-1,1].
\end{equation}
Moreover, (weak) solutions to system~\eqref{eq:detsys} satisfy $\rhoa\ge0$ almost everywhere in $\Om\times(0,T)$,
see, e.g., \cite{Biler1992}, Proposition 1, and the references therein.
\par
We observe that for $\calP=\de_1(d\al)$, system~\eqref{eq:detsys} reduces to the classical Keller-Segel system
for a single population, denoted by $\psi$:
\begin{equation}
\label{eq:KSbasic}
\left\{
\begin{aligned}
&\de\frac{\pl\psi}{\pl t}=\Delta\psi-\dive(\psi\nabla v),&&\hbox{in\ }\Omega\times(0,T)\\
&\ee\frac{\pl v}{\pl t}=\Delta v+\psi,&&\hbox{in\ }\Omega\times(0,T)\\
&\nu\cdot(\nabla\psi-\psi\nabla v)=0,\quad v=0, &&\hbox{on\ }\pl\Omega\times(0,T)\\
&\psi(x,0)=\psi^0(x),\quad v(x,0)=v^0(x),\quad \psi^0,v^0\ge0,&&\hbox{in\ }\Omega.
\end{aligned}
\right.
\end{equation}
For the sake of future reference, we also explicitly note the two-species case
$\calP(d\al)=\tau\de_{\al_1}(d\al)+(1-\tau)\de_{\al_2}$,
$0<\tau<1$, $\al_1,\al_2\in[-1,1]$.
In this case system~\eqref{eq:detsys} takes the form:
\begin{equation*}
\left\{
\begin{aligned}
&\de_1\frac{\pl\rho_1}{\pl t}=\Delta\rho_1-\dive(\al_1\rho_1\nabla v),
&&\hbox{in\ }\Omega\times(0,T),\\
&\de_2\frac{\pl\rho_2}{\pl t}=\Delta\rho_2-\dive(\al_2\rho_2\nabla v),
&&\hbox{in\ }\Omega\times(0,T),\\
&\ee\frac{\pl v}{\pl t}=\Delta v+\tau\al_1\rho_1+(1-\tau)\al_2\rho_2,
&&\hbox{in\ }\Omega\times(0,T)\\
&\nu\cdot(\nabla\rho_1-\al_1\rho_1\nabla v)=0
=\nu\cdot(\nabla\rho_2-\al_2\rho_2\nabla v),
&&\hbox{on\ }\pl\Omega\times(0,T)\\
&v=0,
&&\hbox{on\ }\pl\Omega\times(0,T)\\
&\rho_1(x,0)=\rho_1^0(x)\ge0,\quad \rho_2(x,0)=\rho_2^0(x)\ge0
&&\hbox{in\ }\Omega\\
&v(x,0)=v^0(x),
&&\hbox{in\ }\Omega.
\end{aligned}
\right.
\end{equation*}
\subsection*{}
System~\eqref{eq:detsys} admits the following relevant limit cases.
\subsubsection*{Slow population dynamics limit: $\de_\al>0$, $\ee=0$}
In this case,
system \eqref{eq:detsys} reduces to the following parabolic-elliptic system:
\begin{equation}
\label{eq:rhosysintro}
\left\{
\begin{aligned}
&\de_\al\frac{\pl\rhoa}{\pl t}=\Delta\rhoa-\al\dive(\rhoa\nabla v),&&\hbox{in\ }\Omega\times(0,T),\ \al\in[-1,1]\\
&-\Delta v=\int_{[-1,1]}\al\rhoa\,\calP(d\al),&&\hbox{in\ }\Omega\times(0,T)\\
&\nu\cdot(\nabla\rhoa-\al\rhoa\nabla v)=0,\quad v=0 &&\hbox{on\ }\pl\Omega\times(0,T)\\
&\rhoa(x,0)=\rhoa^0(x)\ge0,&&\hbox{in\ }\Omega.
\end{aligned}
\right.
\end{equation}
Systems of the form \eqref{eq:rhosysintro} also appear in statistical mechanics
(where they are sometimes called Smoluchowski-Poisson systems)
as well as in the theory of semiconductors, see \cite{Biler1992, Chavanis, Gajewski1985} and the references therein.
In the context of chemotaxis, concentration phenomena for \eqref{eq:rhosysintro} were obtained in \cite{JaegerLuckhaus}.
We note that system~\eqref{eq:rhosysintro} decouples, in the sense that it may be written as an integro-differential system
for the populations $\rhoa$, $\al\in[-1,1]$:
\begin{equation}
\label{eq:rhosys2}
\de_\al\frac{\pl\rhoa}{\pl t}=\Delta\rho-\dive\left(\al\rhoa\nabla\iint_{\Om\times[-1,1]}G(x,y)\beta\rho_\beta(y)\,dy\calP(d\beta)\right),
\qquad\al\in[-1,1],
\end{equation}
where $G$ denotes the Green's function for $-\Delta$, see \eqref{def:G} below for the precise definition.
\subsubsection*{Fast population dynamics limit: $\de_\al=0$ for all $\al\in[-1,1]$,
$\ee=1$}
As already mentioned, we are particularly interested in this case.
Under this limit
we obtain the following elliptic-parabolic system:
\begin{equation}
\label{eq:ellpar}
\left\{
\begin{aligned}
&\Delta\rhoa-\al\dive(\rhoa\nabla v)=0,&&\hbox{in\ }\Omega\times(0,T),\ \al\in[-1,1]\\
&\frac{\pl v}{\pl t}=\Delta v+\int_{[-1,1]}\al\rhoa\,\calP(d\al),&&\hbox{in\ }\Omega\times(0,T)\\
&\nu\cdot(\nabla\rhoa-\al\rhoa\nabla v)=0,\ v=0,&&\hbox{on\ }\pl\Omega\times(0,T)\\
&v(x,0)=v^0(x),&&\hbox{in\ }\Omega.
\end{aligned}
\right.
\end{equation}
In this case, it is not difficult to check
(see the proof of Theorem~\ref{thm:var}-(iii) in Section~\ref{sec:variational} below) that
\[
\rhoa(x,t)=C_\al(t)e^{\al v(x,t)}
\]
for some $C_\al(t)>0$ independent of $x\in\Om$.
Therefore,
system~\eqref{eq:ellpar}
decouples into the following semilinear parabolic non-local equation for the chemical substance~$v$:
\begin{equation}
\label{eq:vintro}
\frac{\pl v}{\pl t}=\Delta v+\int_{[-1,1]}\al C_\al(t) e^{\al v}\,\calP(d\al).
\end{equation}
The limit system \eqref{eq:ellpar} no longer implies the
total mass conservation \eqref{eq:massintro}. Therefore, 
we cannot a priori exclude the dependence of $C_\al$ on the time $t$ and on the index $\al$.
On the other hand, the explicit value of $C_\al(t)$ is irrelevant to the dynamics of $\rhoa$,
which only involves $\nabla v$ by the first equation of \eqref{eq:ellpar}.
Therefore, we \textit{assume} a suitable form of mass conservation.
We focus our attention on the following \textit{average} mass conservation property
with respect to $\calP$:
\begin{equation}
\label{eq:averagemassintro}
\iint_{\Om\times[-1,1]}\rhoa(x,t)\,dx\calP(d\al)=\la,
\qquad\mbox{for all }t\in(0,T).
\end{equation}
As already mentioned, such a \lq\lq relaxed" mass conservation property is natural in the situation
where the single species $\rhoa$ are produced by a cell differentiation process.
From \eqref{eq:vintro}--\eqref{eq:averagemassintro} we finally obtain the following non-local evolution problem
for $v$:
\begin{equation}
\label{eq:vintrofinal}
\left\{
\begin{aligned}
&\frac{\pl v}{\pl t}=\Delta v+\la\int_{[-1,1]}\frac{\al e^{\al v}}{\iint_{\Om\times[-1,1]}e^{\beta v(y,t)}\,dy\calP(d\beta)}\,\calP(d\al),
&&\mbox{in\ }\Om\times(0,T)\\
&v=0,
&&\mbox{on\ }\partial\Om\times(0,T)\\
&v(x,0)=v^0(x),
&&\mbox{in\ }\Omega.
\end{aligned}
\right.
\end{equation}
Interestingly, the exponential type nonlinearity in \eqref{eq:vintrofinal}
is exactly the nonlinearity contained in the mean field equation derived by Neri \cite{Neri}
in the context of the statistical mechanics description of 2D turbulence,
extending Onsager's approach \cite{Onsager}, see also \cite{CLMP}.
The steady states for \eqref{eq:vintrofinal} received a considerable attention in recent years, see, e.g.,
\cite{RiZe2016, DeMarchisRicciardi, RicciardiTakahashi, PistoiaRicciardi, GuiJevnikarMoradifam} and the references therein.
Thus, by analyzing \eqref{eq:vintrofinal}, we provide further insight for the mean field equation derived in \cite{Neri}.
Results for the evolution problems of the \lq\lq mean field" form \eqref{eq:vintrofinal},
in the \lq\lq standard" case  $\calP(d\al)=\de_1(d\al)$
were obtained in \cite{KavallarisSuzuki, Wo1997, BebernesLacey, BebernesTalaga}.
Some related non-local evolution problems have also been analyzed
in connection with the modelling of shear banding and Ohmic heating, see \cite{Lacey1995, Lacey1995II, KLT2004}
and the references therein.
\par
From the mathematical point of view, we are interested in the variational structures associated to the
multi-species chemotaxis species \eqref{eq:detsys}, which are a key tool in establishing the global existence
of solutions \cite{BebernesLacey, GajewskiZacharias1998, QS07, Ho2011}.
In particular, we rigorously establish the existence of a Lyapunov functional and
we establish a duality principle for $\rhoa$ and $v$.
Some of these results are stated and justified heuristically in \cite{Suzuki2008book}.
The rigorous proof however requires some care, since the
natural functional space for $(\rhoa)_{\al\in[-1,1]}$ is the logarithmic space $L^1([-1,1],L\log L(\Om);\calP)$,
which is known to be non-reflexive, see, e.g., \cite{RaoRen, RicciardiSuzuki}.
To this end, we adapt some ideas from \cite{Biler1992, RicciardiSuzuki}.
Finally, in the fast population dynamics limit we determine the critical mass for the global existence of solutions
vs.\ chemotactic collapse \cite{DolbeaultPerthame2004, JaegerLuckhaus},
in the form of an optimal logarithmic Hardy-Littlewood-Sobolev type inequality in the spirit of \cite{Beckner, ShafrirWolansky}.
In view of the duality principle, our inequality is equivalent to the sharp Moser-Trudinger type inequality, \cite{Moser, Trudinger}, obtained in
\cite{RiZe2012} and thus provides a new proof for it.
\par
This article is organized as follows.
In Section~\ref{sec:results} we state our main results.
In Section~\ref{sec:variational} we obtain the Lyapunov functionals for \eqref{eq:detsys}--\eqref{eq:rhosysintro}--\eqref{eq:ellpar}.
Section~\ref{sec:duality} is devoted to the establishment of the duality principle, whilst
in Section~\ref{sec:Beckner} we prove the logarithmic HLS inequality and thus we obtain the critical mass for global existence.
Section~\ref{sec:appA} contains some technical estimates and
in Section~\ref{sec:appB} we provide some concluding remarks on the steady states of \eqref{eq:detsys}.
In particular, we observe that
the two stationary mean field problems of \cite{Neri} and \cite{SawadaSuzuki},
which have been extensively analyzed in recent years, see \cite{DeMarchisRicciardi, GuiJevnikarMoradifam, JevnikarYang, ORS, RiZe2012, RiZe2016,
Suzuki2008book}
and the references therein,
may both be obtained as steady states of
\eqref{eq:detsys} in the fast population dynamics limit, by assuming different conserved population mass constraints.
Hence, we provide a unified point of view for such stationary problems.
\subsection*{Notation}
In what follows, all integrals are taken in the sense of Lebesgue.
When the integration variable is clear from the context, we may omit it.
%%%%%%%%%%%%%%%%%%%%%%%%%%%%%%%%%%%%%%%%%%%%%%%%%%%%%%%%%%%%%%%%%%%%%%%%%%%%%%%
%%%%%%%%%%%%%%%%%%%%%%%%%%%%%%%%%%%%%%%%%%%%%%%%%%%%%%%%%%%%%%%%%%%%%%%%%%%%%%%~~~~~
\section{Statement of the main results}
\label{sec:results}
%%%%%%%%%%%%%%%%%%%%%%%%%%%%%%%%%%%%%%%%%%%%%%%%%%%%%%%%%%%%%%%%%%%%%%%%%%%%%%%
%%%%%%%%%%%%%%%%%%%%%%%%%%%%%%%%%%%%%%%%%%%%%%%%%%%%%%%%%%%%%%%%%%%%%%%%%%%%%%%
In order to state our main results, we define the following functionals
\begin{equation}
\label{def:functionals}
\begin{aligned}
\calL(\oplus\rhoa,v):=&\iint_{\Om\times[-1,1]}\rhoa(\log\rhoa-1)\,dx\calP(d\al)+\frac{1}{2}\int_\Om|\nabla v|^2\,dx\\
&\qquad\qquad\qquad\qquad\qquad\qquad\qquad
-\iint_{\Om\times[-1,1]}\al\rhoa v\,dx\calP(d\al),\\
\calF(\oplus\rhoa):=&\iint_{\Om\times[-1,1]}\rhoa(\log\rhoa-1)\,dx\calP(d\al)\\
\qquad&-\frac{1}{2}\iint_{[-1,1]^2}\al\beta\,\calP(d\al)\calP(d\beta)\iint_{\Om^2}G(x,y)\rhoa(x)\rho_\beta(y)\,dxdy,\\
\calJ(v):=&\frac{1}{2}\int_\Om|\nabla v|^2\,dx-\la\log\left(\iint_{\Om\times[-1,1]}e^{\al v}\,dx\calP(d\al)\right)+\la(\log\la-1),
\end{aligned}
\end{equation}
defined for $\oplus\rhoa\in L^1([-1,1],L\log L(\Om);\calP)$, $\rhoa\ge0$ for all $\al\in[-1,1]$ and for $v\in H_0^1(\Om),$
where, following \cite{Suzuki2008book}, we denote $\oplus\rhoa:=\oplus_{\al\in[-1,1]}\rhoa=(\rhoa)_{\al\in[-1,1]}$.
\par
We recall that the space $L\log L(\Om)$ is defined as
\[
L\log L(\Om)=\left\{\psi\in L^1(\Om):\ \int_\Om|\psi\log|\psi||<+\infty\right\},
\]
and that it may be structured as an Orlicz space with Young function $\Phi(s)=(s+1)\log(s+1)-s$,
see, e.g., \cite{GajewskiZacharias1998, Ho2011, RaoRen}; however, we shall not need this point of view.
\par
For all $\la>0$ we define the following set of admissible functions
\[
\widetilde\Gamma_\la:=\left\{\oplus\rhoa\in L^1([-1,1],L\log L(\Om);\calP):\
\begin{aligned}
&\qquad\rhoa\ge0\ \forall\al\in[-1,1],\\
&\iint_{\Om\times[-1,1]}\rhoa\,dx\calP(d\al)=\la
\end{aligned}
\right\}.
\]
With this notation, our main results may be summarized as follows.
\begin{thm}[Variational structures]
\label{thm:var}
The following properties hold true.
\begin{enumerate}
  \item[(i)]
The functional~$\calL$ is a Lyapunov functional for \eqref{eq:detsys},
in the sense that the function
\[
g_0(t):=\calL(\oplus\rhoa(x,t),v(x,t))
\]
decreases along solutions $(\oplus\rhoa(x,t),v(x,t))$ to  \eqref{eq:detsys}.
Moreover, $g_0$ decreases strictly unless $\rhoa(x,t)=C_\al(t)e^{\al v(x,t)}$
for some $C_\al(t)>0$ independent of $x\in\Om$.
\item[(ii)]
The functional~$\calF$ is a Lyapunov functional for the Smoluchowski-Poisson system \eqref{eq:rhosys2},
in the sense that the function
\[
h_0(t):=\calF(\oplus\rhoa(x,t))
\]
decreases
along solutions $\oplus\rhoa(x,t)$ to \eqref{eq:rhosys2}.
Moreover, $h_0$ decreases strictly away from stationary solutions.
\item[(iii)]
The semilinear parabolic problem~\eqref{eq:vintrofinal} is the gradient flow for $\calJ$.
 \item[(iv)]
The following duality property holds true:
\[
\inf_{\widetilde\Gamma_\la\times H_0^1(\Om)}\calL=\inf_{\widetilde\Gamma_\la}\calF=\inf_{H_0^1(\Om)}\calJ.
\]
\end{enumerate}
\end{thm}
We note that Lyapunov functionals are a key tool in establishing the global existence of solutions,
see \cite{DolbeaultPerthame2004, GajewskiZacharias1998}.
Although property~(iv) is derived heuristically in \cite{Suzuki2008book}, a rigorous proof is rather delicate due to
the non-reflexivity of the Orlicz space $L\log L(\Om)$. Here we overcome this difficulty by
some \textit{ad hoc} truncation arguments, in the spirit of \cite{RicciardiSuzuki}.
\par
Our next result is a sharp logarithmic HLS inequality for the functional~$\calF$
of the type derived in \cite{Beckner, ShafrirWolansky},
which provides the critical total population mass threshold for the global existence of solutions,
see \cite{DolbeaultPerthame2004, GajewskiZacharias1998, KavallarisSouplet}.
\begin{thm}[Sharp logarithmic HLS type inequality]
\label{thm:Beckner}
Suppose that $\supp\calP\cap\{-1,1\}\neq\emptyset$.
Then, the functional $\calF$ is bounded from below on $\widetilde\Gamma_\la$ if and only if $\la\le8\pi$.
\end{thm}
Here, $\supp\calP$ denotes the support of $\calP$, namely
\begin{equation}
\label{def:suppP}
\supp\calP:=\left\{\al\in[-1,1]:\ \calP(U)>0\ \hbox{for all open neighborhoods $U$ containing $\al$}\right\}.
\end{equation}
We observe that in view of the duality property stated in Theorem~\ref{thm:var}-(iv),
the inequality stated in Theorem~\ref{thm:Beckner}
is equivalent to the Moser-Trudinger type inequality \cite{Moser,Trudinger} derived in \cite{RiZe2012} and
given by
\begin{equation}
\label{eq:MTRZ}
\inf_{H_0^1(\Om)}\calJ>-\infty
\qquad
\hbox{if and only if }\la\le8\pi.
\end{equation}
The proof of Theorem~\ref{thm:Beckner} is independent of the results in \cite{RiZe2012},
hence here we also provide an alternative proof of \eqref{eq:MTRZ}.
\par
The remaining part of this article is devoted to the proofs of Theorem~\ref{thm:var}
and of Theorem~\ref{thm:Beckner}.
%%%%%%%%%%%%%%%%%%%%%%%%%%%%%%%%%%%%%%%%%%%%%%%%%%%%%%%%%%%%%%%%%%%%%%%%%%%%%%%
%%%%%%%%%%%%%%%%%%%%%%%%%%%%%%%%%%%%%%%%%%%%%%%%%%%%%%%%%%%%%%%%%%%%%%%%%%%%%%%~~~~
\section{Variational structures and proof of Theorem~\ref{thm:var}-(i)-(ii)-(iii)}
\label{sec:variational}
%%%%%%%%%%%%%%%%%%%%%%%%%%%%%%%%%%%%%%%%%%%%%%%%%%%%%%%%%%%%%%%%%%%%%%%%%%%%%%%
%%%%%%%%%%%%%%%%%%%%%%%%%%%%%%%%%%%%%%%%%%%%%%%%%%%%%%%%%%%%%%%%%%%%%%%%%%%%%%%
Henceforth, it will be convenient to denote
$I:=[-1,1]$
and to adopt the product space notation introduced in \cite{Neri}.
Namely, let
\begin{equation*}
\begin{aligned}
\tO:=\IO,
\qquad \tx:=(x,\al),
\qquad
\dtx:=dx\calP(d\al).
\end{aligned}
\end{equation*}
We denote
\[
\rho(\tx)=\rho(x,\al):=\rhoa(x).
\]
\subsubsection*{The full system~\eqref{eq:detsys} and the proof of Theorem~\ref{thm:var}-(i)}
In product space notation system \eqref{eq:detsys} takes the form:
\begin{equation}
\label{eq:tsys}
\left\{
\begin{aligned}
&\de_\al\frac{\pl\rho}{\pl t}=\Delta\rho-\al\dive(\rho\nabla v),&&\hbox{in\ }\tO\times(0,T)\\
&\ee\frac{\pl v}{\pl t}=\Delta v+\int_{I}\al\rho\,\calP(d\al),&&\hbox{in\ }\Omega\times(0,T)\\
&\nu\cdot(\nabla\rho-\al\rho\nabla v)=0,\ v=0, &&\hbox{on\ }\pl\Omega\times I\times (0,T)\\
&\rho(\tx,0)=\rho^0(\tx)\ge0,&&\hbox{in\ }\tO\\
&v(x,0)=v^0(x),&&\hbox{in\ }\Om.
\end{aligned}
\right.
\end{equation}
For $\rho\in L\log L(\tO)$, $\rho\ge0$ a.e.\ in $\tO$, and $v\in H_0^1(\Om)$, the functional $\calL$
defined in
\eqref{def:functionals} takes the form:
\begin{equation}
\label{def:L}
\calL(\rho,v)=\int_{\tO}\rho(\tx)(\log\rho(\tx)-1)\,\dtx+\frac{1}{2}\int_{\Om}|\nabla v|^2\,\dtx-\int_{\tO}\al\rho(\tx)v(x)\,\dtx.
\end{equation}
A formal proof of  Theorem~\ref{thm:var}-(i) is easily obtained by straightforward differentiation.
Indeed, for any $\varphi\in L^\infty(\tO)$, we note that formally (and rigorously, if the strict inequality $\rho>0$ holds true)
\begin{equation}
\begin{aligned}
\bra\calL_\rho(\rho,v),\varphi\ket_{L^2(\tO)}
=\int_{\tO}(\log\rho-\al v)\varphi\,\dtx,
\end{aligned}
\end{equation}
where $\bra\calL_\rho(\rho,v),\varphi\ket_{L^2(\tO)}=\frac{d}{ds}\calL(\rho+s\varphi,v)\vert_{s=0}$
denotes the usual G\^{a}teaux derivative.
In particular, along a solution $(\rho(\tx,t),v(x,t))$ to \eqref{eq:detsys} we formally have:
\begin{equation}
\begin{aligned}
\bra\calL_\rho(\rho,v),\rho_t\ket=&\int_{\tO}(\log\rho-\al v)\rho_t\,\dtx
=\int_{\tO}\frac{1}{\de_\al}(\log\rho-\al v)\dive(\rho\nabla(\log\rho-\al v))\,\dtx\\
=&\int_I\frac{\calP(d\al)}{\de_\al}\int_\Om(\log\rho-\al v)\dive(\rho\nabla(\log\rho-\al v))\,dx\\
=&-\int_{\tO}\frac{\rho}{\de_\al}|\nabla(\log\rho-\al v)|^2\,\dtx\le0.
\end{aligned}
\end{equation}
Similarly, for $\xi\in H_0^1(\Om)$ we compute:
\begin{equation*}
\bra\calL_v(\rho,v),\xi\ket=\int_{\tO}(\nabla v\cdot\nabla\xi-\al\rho\xi)\,\dtx
=-\int_{\tO}(\Delta v+\al\rho)\xi\,\dtx.
\end{equation*}
In particular, along a solution $(\rho(\tx,t),v(x,t))$ to \eqref{eq:detsys} we have:
\begin{equation*}
\begin{aligned}
\bra\calL_v(\rho,v),v_t\ket=&-\frac{1}{\ee}\int_{\tO}(\Delta v+\al\rho)\left(\Delta v+\int_I\al'\rho\,\calP(d\al')\right)\,\dtx\\
=&-\frac{1}{\ee}\int_\Om\left(\Delta v+\int_I\al\rho\,\calP(d\al)\right)^2\,dx\le0.
\end{aligned}
\end{equation*}
Thus, along solutions of \eqref{eq:detsys} we formally have the \textit{non-increase of $\calL$:}
\begin{equation}
\label{eq:Lnonincrease}
\frac{d}{dt}\calL(\rho(\tx,t),v(x,t))\le0
\qquad\hbox{for all }t\in(0,T).
\end{equation}
We now provide a rigorous proof of Theorem~\ref{thm:var}-(i),
by adapting an argument in \cite{Biler1992}.
\begin{proof}[Proof of Theorem~\ref{thm:var}-(i)]
Let $(\rho(x,t),v(x,t))$ be a fixed classical solution for \eqref{eq:detsys} and for $\de>0$ let
\[
g_\de(t):=\calL(\rho(x,t)+\de,v(x,t)).
\]
Then,
\[
g_\de(t)-g_\de(0)=\int_0^t\left\{\bra\calL_\rho(\rho+\de,v),\rho_t\ket
+\bra\calL_v(\rho+\de,v),v_t\ket\right\}.
\]
We compute, recalling that in product space notation $\rho=\rho(\tx)=\rho(x,\al)$:
\[
\begin{aligned}
\bra\calL_\rho(\rho+\de,v),\rho_t\ket
=&\int_{\tO}(\log(\rho+\de)-\al v)\,\rho_t\,\dtx
=\int_{\tO}\frac{1}{\de_\al}(\log(\rho+\de)-\al v)\dive(\nabla\rho-\al\rho\nabla v)\\
=&-\int_{\tO}\frac{1}{\de_\al}\nabla(\log(\rho+\de)-\al v)\cdot(\nabla\rho-\al\rho\nabla v)\\
=&-\int_{\tO}\frac{1}{\de_\al}\nabla(\log(\rho+\de)-\al v)\cdot(\nabla\rho-\al(\rho+\de)\nabla v+\al\de\nabla v)\\
=&-\int_{\tO}\frac{\rho+\de}{\de_\al}|\nabla(\log(\rho+\de)-\al v)|^2
-\de\int_{\tO}\frac{1}{\de_\al}\nabla(\log(\rho+\de)-\al v)\cdot\al\nabla v.
\end{aligned}
\]
Using the elementary identity
\begin{equation}
\label{eq:xiid}
|\nabla\log(\rho+\de)|^2=|\nabla(\log(\rho+\de)-\xi)|^2+|\nabla\xi|^2+2\nabla(\log(\rho+\de)-\xi)\cdot\nabla\xi,
\end{equation}
with $\xi=\al v$,
we may write
\[
\nabla(\log(\rho+\de)-\al v)\cdot\al\nabla v
=\frac{1}{2}\{|\nabla\log(\rho+\de)|^2-|\nabla(\log(\rho+\de)-\al v)|^2-|\al\nabla v|^2\}.
\]
We deduce that
\[
\begin{aligned}
\bra\calL_\rho(\rho+\de,v),\rho_t\ket
=-\int_{\tO}\frac{1}{\de_\al}(\rho+\frac{\de}{2})|\nabla(\log(\rho+\de)-\al v)|^2-\frac{\de}{2}\int_{\tO}\frac{1}{\de_\al}|\nabla\log(\rho+\de)|^2
+\frac{\de}{2}\int_{\tO}\frac{\al^2}{\de_\al}|\nabla v|^2.
\end{aligned}
\]
On the other hand, we have
\[
\begin{aligned}
\bra\calL_v(\rho+\de,v),v_t\ket
=&\int_{\tO}\nabla v\cdot\nabla v_t-\int_{\tO}\al(\rho+\de)v_t
=-\int_{\tO}(\Delta v+\al\rho)v_t-\de\int_{\tO}\al v_t\\
=&-\int_{\Om}(\Delta v+\int_I\al\rho\,\calP(d\al))v_t-\de\int_{I}\al\,\calP(d\al)\int_\Om v_t\\
=&-\frac{1}{\ee}\int_{\Om}(\Delta v+\int_I\al\rho\,\calP(d\al))^2-\de\int_{I}\al\,\calP(d\al)\int_\Om v_t.
\end{aligned}
\]
It follows that

\[
\begin{aligned}
g_\de(t)-g_\de(0)
=&-\int_0^t\int_{\tO}\frac{1}{\de_\al}\left(\rho+\frac{\de}{2}\right)|\nabla(\log(\rho+\de)-\al v)|^2
-\frac{\de}{2}\int_0^t\int_{\tO}\frac{1}{\de_\al}|\nabla\log(\rho+\de)|^2\,\dtx\\
&+\frac{\de}{2}\int_0^t\int_I\frac{\al^2}{\de_\al}\calP(d\al)\int_\Om|\nabla v|^2
-\frac{1}{\ee}\int_0^t\int_\Om\left(\Delta v+\int_I\al\rho\,\calP(d\al)\right)^2\\
&-\de\int_0^t\int_I\al\calP(d\al)\int_\Om v_t.
\end{aligned}
\]

We conclude that
\[
\begin{aligned}
g_\de(t)-g_\de(0)
&+\int_0^t\int_{\tO}\frac{1}{\de_\al}\left(\rho+\frac{\de}{2}\right)|\nabla(\log(\rho+\de)-\al v)|^2\\
\le&\frac{\de}{2}\int_0^t\int_I\frac{\al^2}{\de_\al}\calP(d\al)\int_\Om|\nabla v|^2
-\de\int_0^t\int_I\al\calP(d\al)\int_\Om v_t.
\end{aligned}
\]
By continuity of the function $s\mapsto s\log s$ at $0$,
we have
\[
\lim_{\de\to0^+}g_\de(t)=\calL(\rho(\tx,t),v(x,t)).
\]
Therefore, letting $\de\to0^+$ we obtain
\[
\calL(\rho(\tx,t),v(x,t))-\calL(\rho(\tx,0),v(x,0))
+\limsup_{\de\to0^+}\int_0^t\int_{\tO}\frac{1}{\de_\al}(\rho+\frac{\de}{2})|\nabla(\log(\rho+\de)-\al v)|^2\le0.
\]
Hence, the asserted decreasing properties of $\calL$ are established.
\end{proof}
%%%%%%%%%%%%%%%%%%%%%%%%%%%%%%%%%%%%%%%%%%%%%%%%%%%%%%%%%%%%%%%%%%%%%%%%%
%%%%%%%%%%%%%%%%%%%%%%%%%%%%%%%%%%%%%%%%%%%%%%%%%%%%%%%%%%%%%%%%%%%%%%%%%%%%%%%%%%%%%%%%%%%%%%%%%
\subsubsection*{The case $\de_\al>0$, $\ee=0$ and the proof of Theorem~\ref{thm:var}-(ii)}
In product space notation, system~\eqref{eq:rhosysintro} takes the form
\begin{equation}
\label{eq:rhosys}
\left\{
\begin{aligned}
&\de_\al\frac{\pl\rho}{\pl t}=\Delta\rho-\al\dive(\rho\nabla v),&&\hbox{in\ }\tO\times(0,T)\\
&-\Delta v=\int_I\al\rho\,\calP(d\al),&&\hbox{in\ }\Omega\times(0,T)\\
&\nu\cdot(\nabla\rhoa-\al\rhoa\nabla v)=0,\ v=0, &&\hbox{on\ }\pl\Omega\times(0,T),\ \al\in[-1,1]\\
&\rho(\tx,0)=\rho^0(\tx)\ge0,&&\hbox{in\ }\tO.
\end{aligned}
\right.
\end{equation}
We first recall that the Green function $G(\cdot,\cdot)$ for $-\Delta$ in $\Om$
with Dirichlet boundary conditions is defined for $x,y\in\Om$, $x\neq y$, by
\begin{equation}
\label{def:G}
\begin{cases}
-\Delta_xG(x,y)=\de_y,&\mbox{in\ }\Om\\
G(\cdot,y)=0,&\mbox{on\ }\pl\Om.
\end{cases}
\end{equation}
By means of $G$ we may define a symmetric kernel $\tG(x,y,\al,\beta)$ for
$(x,y,\al,\beta)\in\tO\times\tO$, $x\neq y$,
with corresponding convolution operator
defined by
\begin{equation}
\label{def:tG}
(\tG\ast\rho)(x,\al)=\int_{\tO}G(x,y)\rho(y,\beta)\,dy\calP(d\beta).
\end{equation}
We note that we may write:
\begin{equation*}
\begin{aligned}
\int_{\tO}\al\rho\,\tG\ast(\al\rho)\,\dtx
=&\int_{\tO}\al\rho(x,\al)\int_{\tO}G(x,y)\beta\rho(y,\beta)\,dy\calP(d\beta)\\
=&\iint_{\tO^2}\al\beta\,G(x,y)\rho(x,\al)\rho(y,\beta)\,dxdy\calP(d\al)\calP(d\beta).
\end{aligned}
\end{equation*}
Therefore, the functional $\calF$ may be equivalently written in the form
\begin{equation*}
\calF(\rho):=\int_{\tO}\rho(\log\rho-1)-\frac{1}{2}\int_{\tO}\al\rho\,\tG\ast(\al\rho).
\end{equation*}
For later use, we observe that we may also write:
\begin{equation}
\label{eq:Gastrho}
\int_{\tO}\al\rho\,\tG\ast(\al\rho)\,\dtx
=\int_\Om\left(\int_I\al\rho\,\calP(d\al)\right)G\ast\left(\int_I\al\rho\,\calP(d\al)\right)\,dx.
\end{equation}
From \eqref{eq:rhosys} we deduce that
\[
v=\tG\ast(\al\rho)=G\ast\left(\int_I\al\rho\,\calP(d\al)\right).
\]
%Thus, we compute, recalling that $\rho>0$ in view of the maximum principle:
%\[
%\begin{aligned}
%\frac{d}{dt}\calF(\rho)=&\int_{\tO}\rho_t\log\rho
%-\frac{1}{2}\int_{\tO}\al\rho_t\,\tG\ast(\beta\rho)
%-\frac{1}{2}\int_{\tO}\al\rho\,\tG\ast(\beta\rho_t)\\
%=&\int_{\tO}\rho_t\log\rho
%-\int_{\tO}\al\,\tG\ast(\beta\rho)\,\rho_t{eq:Gastrho}
%=\int_{\tO}\left(\log\rho-\al\,\tG\ast(\beta\rho)\right)\,\rho_t,
%\end{aligned}
%\]
%where we used symmetry of $\tG$ and relabelling of the variables to derive the last line.
%Now, the equation for $\rho$ in \eqref{eq:rhosys} implies that
%\[
%\begin{aligned}
%\frac{d}{dt}\calF(\rho)=&\int_{\tO}\left(\log\rho-\al\,\tG\ast(\beta\rho)\right)
%\frac{1}{\de_\al}\dive\left\{\rho\nabla\left(\log\rho-\al\,\tG\ast(\beta\rho)\right)\right\}\\
%=&-\int_{\tO}\frac{\rho}{\de_\al}\left|\nabla\left(\log\rho-\al\,\tG\ast(\beta\rho)\right)\right|^2\le0,
%\end{aligned}
%\]
%as asserted.
\begin{proof}[Proof of Theorem~\ref{thm:var}-(ii)]
Similarly as above, for $\de>0$ let
\[
h_\de(t):=\calF(\rho(\tx,t)+\de).
\]
Then, using the symmetry of $\tG$, we compute
\[
\begin{aligned}
h_\de'(t)=&\int_{\tO}\left\{\log(\rho+\de)-\al\tG\ast(\al(\rho+\de))\right\}\rho_t\\
=&\int_{\tO}\frac{1}{\de_\al}\left\{\log(\rho+\de)-\al\tG\ast(\al(\rho+\de))\right\}
\dive(\nabla\rho-\al\rho\nabla\tG\ast(\al\rho))\\
=&-\int_{\tO}\frac{1}{\de_\al}\nabla\left\{\log(\rho+\de)-\al\tG\ast(\al\rho))\right\}
\cdot\{\nabla\rho-\al(\rho+\de)\nabla\tG\ast(\al\rho)+\al\de\nabla\tG\ast(\al\rho)\}\\
&-\de\int_{\tO}\frac{\al}{\de_\al}\nabla\tG\ast\al\cdot\{\nabla\rho-\al\rho\nabla\tG\ast(\al\rho)\}\\
=&-\int_{\tO}\frac{\rho+\de}{\de_\al}|\nabla\{\log(\rho+\de)-\al\tG\ast(\al\rho)\}|^2-I-II
\end{aligned}
\]
where
\[
\begin{aligned}
I:=&\int_{\tO}\frac{\de}{\de_\al}\nabla\{\log(\rho+\de)-\al\tG\ast(\al\rho)\}\cdot\al\nabla\tG\ast(\al\rho),\\
II:=&\de\int_{\tO}\frac{\al}{\de_\al}\nabla\tG\ast\al\cdot\{\nabla\rho-\al\rho\nabla\tG\ast(\al\rho)\}.
\end{aligned}
\]
Using \eqref{eq:xiid} with $\xi=\al\tG\ast(\al\rho)$, we have
\[
\begin{aligned}
\nabla\{\log(\rho+\de)-\al\tG\ast(\al\rho)\}\cdot\al\nabla\tG\ast(\al\rho)
=&\frac{1}{2}|\nabla\log(\rho+\de)|^2-\frac{1}{2}|\nabla(\log(\rho+\de)-\al\tG\ast(\al\rho)|^2\\
&\qquad\qquad-\frac{1}{2}|\nabla\al\tG\ast(\al\rho)|^2.
\end{aligned}
\]
Therefore,
\[
\begin{aligned}
h_\de'(t)=-\int_{\tO}\frac{1}{\de_\al}\left(\rho+\frac{\de}{2}\right)|\nabla(\log(\rho+\de)&-\al\tG\ast(\al\rho))|^2
-\frac{\de}{2}\int_{\tO}\frac{1}{\de_\al}|\nabla\log(\rho+\de)|^2\\
&\qquad+\frac{\de}{2}\int_{\tO}\frac{1}{\de_\al}|\nabla\al\tG\ast(\al\rho)|^2-II
\end{aligned}
\]
We conclude that
\[
\begin{aligned}
h_\de(t)-h_\de(0)+&\int_0^t\int_{\tO}\frac{1}{\de_\al}\left(\rho+\frac{\de}{2}\right)|\nabla(\log(\rho+\de)-\al\tG\ast(\al\rho))|^2\\
\le&\frac{\de}{2}\int_0^t\int_{\tO}\frac{1}{\de_\al}|\nabla\al\tG\ast(\al\rho)|^2
-\de\int_0^t\int_{\tO}\al\nabla\tG\ast\al\cdot\{\nabla\rho-\al\rho\nabla\tG\ast(\al\rho)\}.
\end{aligned}
\]
Now we observe that $\lim_{\de\to0^+}h_\de(t)=\calF(\rho(\tx,t))$.
Therefore, letting $\de\to0^+$, we obtain
\[
\calF(\rho(\tx,t))-\calF(\rho(\tx,0))
+\limsup_{\de\to0^+}\int_0^t\int_{\tO}\frac{1}{\de_\al}\left(\rho+\frac{\de}{2}\right)|\nabla(\log(\rho+\de)-\al\tG\ast(\al\rho))|^2\le0,
\]
and the asserted monotonicity property for $\calF(\rho(\tx,t))$ follows.
\par
If the decrease is not strict, then $\nabla(\log\rho-\al\tG\ast(\al\rho))\equiv0$.
In view of \eqref{eq:rhosys2}, we conclude that the solution is stationary.
\end{proof}
%%%%%%%%%%%%%%%%%%%%%%%%%%%%%%%%%%%%%%%%%%%%%%%%%%%%%%%%%%%%%%%%%%%%%%%%%%%%%%%%%%%%%%%%%%%%%%%%%
\subsubsection*{The case $\de_\al=0$, $\ee=1$ and the proof of  Theorem~\ref{thm:var}-(iii)}
In product space notation system~\eqref{eq:ellpar} takes the form
\begin{equation}
\label{eq:sys2}
\left\{
\begin{aligned}
&\Delta\rho-\al\dive(\rho\nabla v)=0,&&\hbox{in\ }\tO\times(0,T)\\
&\frac{\pl v}{\pl t}=\Delta v+\int_I\al\rho\,\calP(d\al),&&\hbox{in\ }\Omega\times(0,T)\\
&\nu\cdot(\nabla\rho-\al\rho\nabla v)=0,\ v=0, &&\hbox{on\ }\pl\Omega\times I\times(0,T)\\
&v(x,0)=v^0(x),&&\hbox{in\ }\Omega.
\end{aligned}
\right.
\end{equation}
\begin{proof}[Proof of  Theorem~\ref{thm:var}-(iii)]
We observe that for every fixed $\al\in I$, $t\in(0,T)$ we may write
\begin{equation}
\label{eq:rhovid}
\nabla\rho-\al\rho\nabla v=e^{\al v}\nabla(e^{-\al v}\rho).
\end{equation}
Multiplying the first equation in \eqref{eq:sys2} by $e^{-\al v}\rho$ and integrating,
in view of the no-flux boundary condition, we have:
\begin{equation*}
\begin{aligned}
0=\int_{\pl\Omega}e^{-\al v}\rho\nu\cdot(\nabla\rho-\al\rho\nabla v)-\int_\Om e^{\al v}|\nabla(e^{-\al v}\rho)|^2
=-\int_\Om e^{\al v}|\nabla(e^{-\al v}\rho)|^2.
\end{aligned}
\end{equation*}
We deduce that $\nabla(e^{-\al v}\rho)=0$ a.e.\ in $\Omega$, and consequently
\begin{equation}
\label{eq:rhoexpv}
\rho(x,\al,t)=C_\al(t) e^{\al v(x,t)}
\end{equation}
for some $C_\al(t)\ge0$.
We shall assume that $C_\al(t)$ is independent of $\al$. We note that such an assumption does not affect the dynamics of
the population species $\rho$, which only depends on $\nabla v.$
Assuming the mass conservation~\eqref{eq:averagemassintro}, we derive from \eqref{eq:sys2}--\eqref{eq:rhoexpv}
the following evolution problem for $v$:
\begin{equation}
\label{eq:MFevol}
\left\{
\begin{aligned}
&\frac{\pl v}{\pl t}=\Delta v+\la\frac{\int_I\al e^{\al v}\,\calP(d\al)}{\iint_{\IO}e^{\al v}\,\calP(d\al)dx},
&&\hbox{in\ }\Om\times(0,T)\\
&v(x,t)=0,&&\hbox{on }\pl\Om\times(0,T)\\
&v(x,0)=v^0(x),&&\hbox{in\ }\Om.
\end{aligned}
\right.
\end{equation}
We recall from \eqref{def:functionals} that
\begin{equation*}
\calJ(v)=\frac{1}{2}\int_\Om|\nabla v|^2\,dx-\la\log\int_{\tO}e^{\al v}\,\dtx+\la(\log\la-1),
\qquad v\in H_0^1(\Om).
\end{equation*}
It is readily checked that \eqref{eq:MFevol} is the gradient flow for $\mathcal J_\la$.
\end{proof}
%%%%%%%%%%%%%%%%%%%%%%%%%%%%%%%%%%%%%%%%%%%%%%%%%%%%%%%%%%%%%%%%%%%%%%%%%%%%%%%~~~~~
%%%%%%%%%%%%%%%%%%%%%%%%%%%%%%%%%%%%%%%%%%%%%%%%%%%%%%%%%%%%%%%%%%%%%%%%%%%%%%%~~~~~
\section{Duality and Proof of Theorem~\ref{thm:var}-(iv)}
\label{sec:duality}
%%%%%%%%%%%%%%%%%%%%%%%%%%%%%%%%%%%%%%%%%%%%%%%%%%%%%%%%%%%%%%%%%%%%%%%%%%%%%%%~~~~~
%%%%%%%%%%%%%%%%%%%%%%%%%%%%%%%%%%%%%%%%%%%%%%%%%%%%%%%%%%%%%%%%%%%%%%%%%%%%%%%~~~~~
We recall from \eqref{def:functionals} that $\calL$ is defined for $\rho\in L\log L(\tO)$, $\rho\ge0$, and $v\in H_0^1(\Om)$ by 
\begin{equation*}
\calL(\rho,v):=\int_{\tO}\rho(\log\rho-1)\,\dtx
+\frac{1}{2}\int_\Om|\nabla v|^2\,dx-\int_{\tO}\al\rho v\,\dtx
\end{equation*}
and
\begin{equation*}
{\tilde\Gamma_\la}:=\left\{\rho\in L\log L(\tO):\ \rho\ge0\ \mathrm{a.e.\ in}\ \tO,\ \int_{\tO}\rho(\tx)\,\dtx=\la\right\}.
\end{equation*}
The main properties needed to establish Theorem~\ref{thm:var}-(iv) are contained in the following statement.
\begin{prop}
\label{prop:minimizer}
For every fixed $v\in H_0^1(\Om)$ there exists $\rho_v\in\tA$ such that
\begin{equation*}
\inf_{\tA}\calL(\cdot,v)=\calL(\rho_v,v).
\end{equation*}
Moreover, $\rho_v$ satisfies
\begin{equation}
\label{eq:rhovform}
\rho_v=\la\frac{e^{\al v}}{\int_{\tO}e^{\al v}\,\dtx},\qquad \hbox{a.e.\ in\ }\tO.
\end{equation}
\end{prop}
Before we proceed further with the proof of Proposition~\ref{prop:minimizer} 
we need to state and prove two auxiliary results. We first point out that a minimizing sequence $\rho_n\in\tA$ for $\calL(\cdot,v)$ may be taken
uniformly bounded in $L^\infty(\tO)$ and
moreover the minimizer $\rho_v$ satisfies $\rho_v>0$ a.e.\ in $\tO$, following an approach established  in \cite{RicciardiSuzuki}.
The underlying idea is that, since the nonlinearity
\begin{equation}
\label{def:f}
f(t)=t(\log t-1)
\end{equation}
blows up at infinity and attains a strictly negative minimum given by $\min f=f(1)=-1$,
the minimizing sequence $\rho_n$ may be modified so that
$0\le\rho_n\le M$ for some $M>0$ independent of $n$, a.e.\ in $\tO$, without increasing the value of $\calL(\cdot,v)$,
and the minimizer $\rho_v$ satisfies $\rho_v>0$ a.e.\ in $\tO$.
Then, the proof of Proposition~\ref{prop:minimizer}  easily follows.
\begin{lemma}
\label{lem:Linftybound}
For any fixed $v\in H_0^1(\Om)\cap L^\infty(\Om)$ there exists $M>0$ depending only on $\tO,\la$ and $v$
such that for any $\rho\in\tA$
there exists $\rho^*\in\tA$
such that $0\le\rho^*\le M$ and
\[
\calL(\rho^*,v)\le\calL(\rho,v).
\]
\end{lemma}
\begin{proof}
For a fixed $M>2\la/|\Om|$ we define:
\[
\tilA:=\{\tx\in\tO:\ \rho\ge M\},
\qquad
\tilE:=\left\{\tx\in\tO:\ \rho\le \frac{2\la}{|\Om|}\right\},
\qquad
\km:=\int_{\tilA}(\rho-M).
\]
We claim that
\begin{equation}
\label{eq:Elb}
|\tilE|\ge\frac{|\Om|}{2}.
\end{equation}
Indeed, we have:
\[
\la=\int_{\tO}\rho\,\dtx=\int_{\tilE}\rho\,\dtx+\int_{\tO\setminus\tilE}\rho\,\dtx\ge\frac{2\la}{|\Om|}(|\tO|-|\tilE|)
=2\la\left(1-\frac{|\tilE|}{|\tO|}\right),
\]
where we used the fact $|\tO|=\calP(I)|\Om|=|\Om|$.
This implies \eqref{eq:Elb}.
\par
We also note that $\km\le\la$ and therefore, in view of \eqref{eq:Elb}:
\begin{equation}
\label{eq:kmest}
\frac{\km}{|\tilE|}\le\frac{2\la}{|\Om|}.
\end{equation}
We define:
\begin{equation}
\rho^*:=M\chi_{\tilA}+\rho\chi_{\tO\setminus(\tilA\cup\tilE)}+\left(\rho+\frac{\km}{|\tilE|}\right)\chi_{\tilE}.
\end{equation}
It is readily checked that $\rho^*\in\tA$, indeed we have:
\begin{equation*}
\begin{aligned}
\int_{\tO}\rho^*=&M|\tilA|+\int_{\tO\setminus(\tilA\cup\tilE)}\rho\,\dtx+\int_{\tilE}\rho\,\dtx+\km\\
=&M|\tilA|+\int_{\tO\setminus\tilA}\rho\,\dtx+\int_{\tilA}(\rho-M)\,\dtx=\int_{\tO}\rho\,\dtx=\la.
\end{aligned}
\end{equation*}
We write:
\[
\begin{aligned}
\calL(\rho^*,v)-\calL(\rho,v)=\int_{\tilA}[f(M)-f(\rho)]
+\int_{\tilE}\left[f\left(\rho+\frac{\km}{|\tilE|}\right)-f(\rho)\right]
-\int_{\tO}\al(\rho^*-\rho)v.
\end{aligned}
\]
Using the Mean Value Theorem, we estimate:
\[
\begin{aligned}
\int_{\tilA}[f(\rho)-f(M)]=\int_{\tilA}f'(M+\theta(x)(\rho-M))(\rho-M)
\ge\log M\int_{\tilA}(\rho-M)=\km\log M,
\end{aligned}
\]
where $0\le\theta(x)\le1$.
Similarly, we have
\begin{equation}
\label{eq:tilEest}
\int_{\tilE}\left[f\left(\rho+\frac{\km}{|\tilE|}\right)-f(\rho)\right]\le\km C(f,\la),
\end{equation}
where $C(f,\la)=\max_{1/2\le s\le4\la/|\Om|}|f'(s)|$.
Indeed, since $f$ is decreasing on $[0,1]$, if $\km/|\tilE|\le1/2$,
we readily have
\[
\int_{\tilE\cap\{0\le\rho\le1/2\}}\left[f(\rho)-f\left(\rho+\frac{\km}{|\tilE|}\right)\right]\ge0.
\]
If $\km/|\tilE|\ge1/2$, then $0\le\rho+\km/|\tilE|-1/2\le\km/|\tilE|$
and therefore
\[
\begin{aligned}
\int_{\tilE\cap\{0\le\rho\le1/2\}}&
\left[f(\rho)-f\left(\rho+\frac{\km}{|\tilE|}\right)\right]
\ge\int_{\tilE\cap\{0\le\rho\le1/2\}}\left[f\left(\frac{1}{2}\right)-f\left(\rho+\frac{\km}{|\tilE|}\right)\right]\\
&=\int_{\tilE\cap\{0\le\rho\le1/2\}}f'\left(\frac{1}{2}+\theta(x)\left(\rho+\frac{\km}{|\tilE|}-\frac{1}{2}\right)\right)\left(\rho+\frac{\km}{|\tilE|}-\frac{1}{2}\right)\\
&\ge-\km\max_{1/2\le s\le1/2+2\la/|\Om|}|f'(s)|.
\end{aligned}
\]
Hence, \eqref{eq:tilEest} is established.
Finally, we have
\[
\begin{aligned}
\left|\int_{\tO}(\rho^*-\rho)\al v\right|\le\left|\int_{\tilA}(\rho^*-\rho)\al v\right|+\left|\int_{\tilE}(\rho^*-\rho)\al v\right|
\le&\int_{\tilA}|\rho-M|\|v\|_\infty+\km\|v\|_\infty\\
\le&2\km\|v\|_\infty.
\end{aligned}
\]
We conclude that
\[
\calL(\rho^*,v)-\calL(\rho,v)\le(-\log M+C(f,\la)+2\|v\|_\infty)\km
=(-\log M+O(1))\km
\]
and the asserted statement follows by letting $M\to+\infty$.
\end{proof}
For $\rho\in\tA$ we define
\[
\tilA:=\left\{\tx\in\tO:\rho(\tx)\ge\frac{\la}{2|\Om|}\right\}
\qquad\mbox{and}\qquad\tilE:=\{\tx\in\tO:\rho(\tx)=0\}.
\]
\begin{lemma}
\label{lem:positivity}
Fix $v\in H_0^1(\Om)\cap L^\infty(\Om)$.
Suppose that $|\tilE|>0$. Then, there exists $\rho_*\in\tA$ such that
$\rho_*>0$ a.e.\ in $\tO$ and
\[
\calL(\rho_*,v)-\calL(\rho,v)<0.
\]
\end{lemma}
\begin{proof}
We claim that $|\tilA|>0$.
Indeed, if it is not the case, we have $\rho\le\la/(2|\Om|)$ a.e.\ in $\tO$.
It follows that
\[
\la=\int_{\tO}\rho\,\dtx\le|\tO|\frac{\la}{2|\Om|}=\frac{\la}{2},
\]
a contradiction.
Thus, we may define
\[
\varphi:=\frac{\chi_{\tilE}}{|\tilE|}-\frac{\chi_{\tilA}}{|\tilA|}
=\begin{cases}
|\tilE|^{-1},&\mathrm{in}\ \tilE\\
0,&\mathrm{in}\ \tO\setminus(\tilA\cup\tilE)\\
-|\tilA|^{-1},&\mathrm{in}\ \tilA.
\end{cases}
\]
For $t>0$ sufficiently small we set
\[
\rho_*:=\rho+t\varphi.
\]
We note that since $\int_{\tO}\varphi\,\dtx=0$, we have $\rho_*\in\tA$.
Using the identity
\[
\int_{\tO}(\rho+t\varphi)(\log(\rho+t\varphi)-1)-\int_{\tO}\rho(\log\rho-1)
=\int_{\tO}\rho(\log(\rho+t\varphi)-\log\rho)+\int_{\tO}t\varphi(\log(\rho+t\varphi)-1),
\]
we may write:
\[
\begin{aligned}
\calL(\rho_*,v)-\calL(\rho,v)
=&\int_{\tO}\rho(\log(\rho+t\varphi)-\log\rho)+\int_{\tO}t\varphi(\log(\rho+t\varphi)-1)-\int_{\tO}\al t\varphi v\\
=&\int_{\tilA}\rho\left(\log\left(\rho-\frac{t}{|\tilA|}\right)-\log\rho\right)
+\int_{\tilE}\frac{t}{|\tilE|}\left(\log\frac{t}{|\tilE|}-1\right)\\
&\qquad-\int_{\tilA}\frac{t}{|\tilA|}\left(\log\left(\rho-\frac{t}{|\tilA|}\right)-1\right)
+\int_{\tilA}\al\frac{t}{|\tilA|}v-\int_{\tilE}\al\frac{t}{|\tilE|}v\\
=&t\left\{\left(\log\frac{t}{|\tilE|}-1\right)+\frac{1}{|\tilA|}\int_{\tilA}\frac{\rho|\tilA|}{t}\log\left(1-\frac{t}{\rho|\tilA|}\right)\right.\\
&\left. \qquad-\frac{1}{|\tilA|}\int_{\tilA}\left[\log\left(\rho-\frac{t}{|\tilA|}\right)-1\right]-\frac{1}{|\tilE|}\int_{\tilE}\al v
+\frac{1}{|\tilA|}\int_{\tilA}\al v\right\}\\
=&t\{\log t+O(1)\}
\end{aligned}
\]
as $t\to0^+$, where in order to derive the last line we used the fact
\[
\frac{|\tilA|}{t}\int_{\tilA}\rho\left(\log\left(\rho-\frac{t}{|\tilA|}\right)-\log\rho\right)
=\int_{\tilA}\frac{|\tilA|\rho}{t}\log\left(1-\frac{t}{|\tilA|\rho}\right)=O(1).
\]
We conclude that $\calL(\rho_*,v)-\calL(\rho,v)<0$ for sufficiently small values of $t>0$.
\end{proof}
\begin{proof}[Proof of Proposition~\ref{prop:minimizer}]
In view of Lemma~\ref{lem:Linftybound}, we may assume that the minimizing sequence $\rho_n$ is uniformly bounded
in $L^\infty(\tO)$. In particular, it is uniformly bounded in $L^p(\tO)$ for all $1<p<+\infty$.
Consequently, there exists $\rho_v\in L^p(\tO)$ such that, up to subsequences, $\rho_n\rightharpoonup\rho_v\in\tilde\Gamma_\la$
weakly in $L^p(\tO)$, for all $1<p<+\infty$.
By convexity of $\calL(\cdot,v)$, $\rho_v$ is the desired minimizer.
We are left to establish \eqref{eq:rhovform}.
To this end, for every $\de>0$ we define
$\Lambda_\de:=\{\rho_v>\de\}$ and $U_\de:=\{\varphi\in L^\infty(\tO):\ \|\varphi\|_{L^\infty(\tO)}<\de/2\}$.
We can differentiate the function $\calL(\rho_v+t\chi_{\Lambda_\de}\varphi,v)$ with respect to $t$
with constraint $\int_{\tO}\chi_{\Lambda_\de}\varphi\,\dtx=0$ at $t=0$.
We thus obtain that
\[
\log\rho_v-\al v=C\qquad\hbox{a.e.\ in\ }\Lambda_\de,
\]
where $C$ is a Lagrange multiplier. Since for $\de'<\de$ we have $\Lambda_{\de'}\supset\Lambda_\de$, we conclude that
$C$ does not depend on $\de$.
Hence, \eqref{eq:rhovform} holds true in $\bigcup_{\de>0}\Lambda_\de$.
In view of Lemma~\ref{lem:positivity}, we have $|\tO\setminus\bigcup_{\de>0}\Lambda_\de|=0$.
\par
Since $\rho_v\in\tA$, we conclude that
\begin{equation}
\label{eq:rhov}
\rho_v=\la\frac{e^{\al v}}{\int_{\tO}e^{\al v}\,\dtx}
\qquad\hbox{a.e.\ in\ }\tO.
\end{equation}
Now the proof of Proposition~\ref{prop:minimizer} is complete.
\end{proof}
\begin{proof}[Proof of Theorem~\ref{thm:var}-(iv)]
We claim that
\begin{equation}
\label{eq:dualityLJ}
\inf_{\tA}\calL(\cdot,v)=\calL(\rho_v,v)=\calJ(v).
\end{equation}
Indeed, from \eqref{eq:rhov} we derive that
\[
\begin{aligned}
\log\rho_v=\al v-\log\int_{\tO}e^{\al v}+\log\la.
\end{aligned}
\]
We compute
\[
\begin{aligned}
\calL(\rho_v,v)
=&\int_{\tO}\rho_v(\log\rho_v-1)+\frac{1}{2}\int_\Om|\nabla v|^2-\int_{\tO}\al\rho_v v\\
=&\int_{\tO}\rho_v\left(\al v-\log\int_{\tO}e^{\al v}+\log\la-1\right)+\frac{1}{2}\int_\Om|\nabla v|^2-\int_{\tO}\al\rho_v v\\
=&\frac{1}{2}\int_\Om|\nabla v|^2-\la\log\int_{\tO}e^{\al v}+\la(\log\la-1)=\calJ(v),
\end{aligned}
\]
where we used $\int_{\tO}\rho_v=\la$ to derive the last line.
Thus, \eqref{eq:dualityLJ} is established.
\par
Similarly, we claim that for every fixed $\rho\in\tA$ there holds
\begin{equation}
\label{eq:dualityLF}
\inf_{H_0^1(\Om)}\calL(\rho,\cdot)=\calF(\rho).
\end{equation}
Indeed, it is standard to check that $\inf_{H_0^1(\Om)}\calL(\rho,\cdot)$ is attained at the solution $v_\rho\in H_0^1(\Om)$
of the following
\[
-\Delta v_\rho=\int_I\al\rho\,\calP(d\al)\quad\hbox{in\ }\Om,
\qquad
v_\rho=0\quad\hbox{on\ }\pl\Om.
\]
We observe that
\[
\int_\Om|\nabla v_\rho|^2=\int_{\Om}(-\Delta v_\rho)v_\rho
=\int_\Om\int_I\al\rho\,\calP(d\al) v_\rho\,dx
=\int_\Om\int_I\al\rho\,\calP(d\al)\,G\ast\int_I\al\rho\,\calP(d\al).
\]
In view of the above and \eqref{eq:Gastrho} we deduce:
\[
\begin{aligned}
\calL(\rho,v_\rho)
=&\int_{\tO}\rho(\log\rho-1)-\frac{1}{2}\int_{\tO}\int_I\al\rho\calP(d\al) v_\rho\\
=&\int_{\tO}\rho(\log\rho-1)-\frac{1}{2}\int_\Om\int_I\al\rho\calP(d\al)G\ast\int_I\al\rho\calP(d\al).
\end{aligned}
\]
Now the proof of Theorem~\ref{thm:var}-(iv) follows from \eqref{eq:dualityLJ} and \eqref{eq:dualityLF}.
\end{proof}
%%%%%%%%%%%%%%%%%%%%%%%%%%%%%%%%%%%%%%%%%%%%%%%%%%%%%%%%%%%%%%%%%%%%%%%%%%%%%%%%%%~~~~~
%%%%%%%%%%%%%%%%%%%%%%%%%%%%%%%%%%%%%%%%%%%%%%%%%%%%%%%%%%%%%%%%%%%%%%%%%%%%%%%%%%~~~~~
\section{Critical mass and proof of Theorem~\ref{thm:Beckner}}
\label{sec:Beckner}
%%%%%%%%%%%%%%%%%%%%%%%%%%%%%%%%%%%%%%%%%%%%%%%%%%%%%%%%%%%%%%%%%%%%%%%%%%%%%%%%%%~~~~~
%%%%%%%%%%%%%%%%%%%%%%%%%%%%%%%%%%%%%%%%%%%%%%%%%%%%%%%%%%%%%%%%%%%%%%%%%%%%%%%%%%~~~~~
In order to prove Theorem~\ref{thm:Beckner} we set
\[
\Gamma_\la=\left\{\psi\in L\log L(\Om):\ \psi\ge0\ \hbox{a.e.\ in $\Om$},\ \int_\Om\psi=\la\right\}.
\]
We recall that $f(t)=t(\log t-1)$ for $t\ge0$, see \eqref{def:f}.
For $\psi\in\Gamma_\la$ let
\[
\calF_0(\psi)=\int_\Om\psi(\log\psi-1)\,dx-\frac{1}{2}\int_\Om\psi\,G\ast\psi\,dx.
\]
The following sharp logarithmic Hardy-Littlewood-Sobolev inequality is due to Beckner.
\begin{lemma}[\cite{Beckner}]
\label{lem:Beckner}
The functional~$\calF_0$ is bounded from below on $\Gamma_\la$ if and only if $\la\le8\pi$.
\end{lemma}
We shall need the following slightly more general result, which follows directly from Lemma~\ref{lem:Beckner}.
\begin{cor}
\label{cor:Beckner}
There holds:
\[
\inf\left\{\calF_0(\psi):\ \psi\in\bigcup_{\la\le8\pi}\Gamma_\la\right\}>-\infty.
\]
\end{cor}
\begin{proof}
Let $\psi\in\Gamma_\la$ and let $0\le t\le1$.
We compute:
\[
\begin{aligned}
\calF_0(t\psi)
=&\int_\Om t\psi(\log(t\psi)-1)-\frac{t^2}{2}\int_\Om\psi\,G\ast\psi
=\int_\Om t\psi(\log\psi+\log t-1)-\frac{t^2}{2}\int_\Om\psi\,G\ast\psi\\
=&t\int_\Om\psi(\log\psi-1)+t\log t\int_\Om\psi-\frac{t^2}{2}\int_\Om\psi\,G\ast\psi\\
=&t\left\{\int_\Om\psi(\log\psi-1)-\frac{t}{2}\int_\Om\psi\,G\ast\psi\right\}+\la\,t\log t.
\end{aligned}
\]
Since $\int_\Om\psi\,G\ast\psi\ge0$, and using the fact $t\log t\ge-e^{-1}$, we deduce that
\[
\calF_0(t\psi)\ge t\calF_0(\psi)-\frac{\la}{e}\ge\min\left\{\inf_{\Gamma_\la}\calF_0,0\right\}-\frac{\la}{e}.
\]
The claim follows.
\end{proof}
\begin{proof}[Proof of Theorem~\ref{thm:Beckner}, \lq\lq if" part]
Setting
\[
\psi_\rho(x):=\left|\int_I\al\rho(x,\al)\,\calP(d\al)\right|
\]
we find that
\begin{equation}
\label{eq:psirhoest}
0\le\psi_\rho(x)\le\int_I\rho(x,\al)\,\calP(d\al)
\end{equation}
and therefore
\[
\int_\Om\psi_\rho\le\int_{\tO}\rho\,\dtx=\la.
\]
In particular, we have
\begin{equation}
\label{eq:psirhoprop}
\psi_\rho\in\bigcup_{\la\le8\pi}\Gamma_\la.
\end{equation}
In view of \eqref{eq:Gastrho} and \eqref{def:f}, we may write
\[
\calF(\rho)=\int_{\tO}f(\rho)-\frac{1}{2}\int_\Om\left(\int_I\al\rho\right)\,\tG\ast\left(\int_I\al\rho\right).
\]
Consequently, we have
\[
\calF(\rho)\ge\int_{\tO}f(\rho)-\frac{1}{2}\int_\Om\psi_\rho\,G\ast\psi_\rho
=\int_{\tO}f(\rho)-\int_\Om f(\psi_\rho)+\calF_0(\psi_\rho).
\]
In view of \eqref{eq:psirhoprop} and Corollary~\ref{cor:Beckner}, we are thus reduced to show that
\begin{equation}
\label{eq:MTsuffcond}
\inf_{\tA}\left\{\int_{\tO}f(\rho)\,\dtx
-\int_\Om f(\psi_\rho)\,dx\right\}>-\infty.
\end{equation}
Since $f$ is convex and $\calP(I)=1$, in view of Jensen's inequality we have, for every \emph{fixed} $x\in\Om$, that
\[
f\left(\int_I\rho(x,\al)\,\calP(d\al)\right)\le\int_If(\rho(x,\al))\,\calP(d\al).
\]
Integrating over $\Om$ we deduce that
\[
\int_\Om f\left(\int_I\rho(x,\al)\,\calP(d\al)\right)dx\le\int_{\tO}f(\rho)\,\dtx.
\]
In order to complete the proof, we observe that from
\eqref{eq:psirhoest} and
some elementary properties of the nonlinearity $f$, in particular the fact
$f(t)\ge-1$ for all $t\ge0$, we obtain
\[
f(\psi_\rho)\le f\left(\int_I\rhoa\,\calP(d\al)\right)+1.
\]
This concludes the proof of the \lq\lq if part" of Theorem~\ref{thm:Beckner}.
\end{proof}
For the proof of the \lq\lq only if" part we may use the same test functions as may be found, e.g., in \cite{RiZe2012}.
For $\eps>0$ let $\Ue$ be the radial \lq\lq Liouville bubble" defined by
\begin{equation}
\label{def:bubble}
\Ue(x):=\log\frac{8\eps^2}{(\eps^2+|x|^2)^2}.
\end{equation}
It is well known that the functions $\Ue$ satisfy
\begin{equation}
\label{eq:Liouville}
\begin{cases}
-\Delta U=e^U&\mbox{in\ }\rr^2\\
\int_{\rr^2}e^U<+\infty,
\end{cases}
\end{equation}
and moreover there holds
\[
\int_{\rr^2}e^{\Ue}=8\pi,
\qquad\mbox{for all }\eps>0.
\]
Without loss of generality we assume that $0\in\Om$.
Let
\begin{equation}
\label{def:psie}
\psie:=\la\frac{e^{\Ue}}{\int_\Om e^{\Ue}}.
\end{equation}
Clearly, $\psie\in\Gamma_\la$ for all $\eps>0$.
We first establish a lemma for the functions $\psie$ defined in \eqref{def:psie}.
\begin{lemma}
\label{lem:psieexp}
The following expansions hold true.
\begin{enumerate}
  \item[(i)]
$\int_\Om\psie\log\psie=\la\log\frac{1}{\eps^2}+O(1)$;
  \item[(ii)]
$\int_\Om\psie\,G\ast\psie=\frac{\la^2}{8\pi+o(1)}\log\frac{1}{\eps^4}+O(1)$.
\end{enumerate}
\end{lemma}
The proof of Lemma~\ref{lem:psieexp} is straightforward; the details are provided in the Appendix.
\par
Now we can conclude the proof of Theorem~\ref{thm:Beckner}.
\begin{proof}[Proof of Theorem~\ref{thm:Beckner}, \lq\lq only if" part]
Assuming that $\la>8\pi$, we provide a family of functions $\rho_\eps\in\widetilde\Gamma_{\lambda}$
such that
\begin{equation}
\label{eq:Finfty}
\calF(\rho_\eps)\to-\infty
\quad\hbox{as\ }\eps\to0^+.
\end{equation}
We assume that $\supp\calP\ni1$, the remaining case being completely analogous.
Let $0<\eta<1$. Then, $\calP([1-\eta,1])>0$.
For all $\eps>0$ we define
\[
\rhoe(\tx)=\rhoe(x,\al):=\la\,\frac{\chi_{[1-\eta,1]}(\al)}{\calP([1-\eta,1])}\,\frac{e^{\Ue(x)}}{\int_\Om e^{\Ue}}
=\frac{\chi_{[1-\eta,1]}(\al)}{\calP([1-\eta,1])}\,\psie(x).
\]
Clearly, $\int_{\tO}\rhoe=\la$ for all $\eps>0$.
\par
We claim that
\begin{equation}
\label{eq:rhoe1}
\int_{\tO}\al\rhoe\,\tG\ast(\al\rhoe)\,\dtx
=\left(\frac{\int_{[1-\eta,1]}\al\,\calP(d\al)}{\calP([1-\eta,1])}\right)^2\int_\Om\psie\,G\ast\psie.
\end{equation}
Indeed, we have:
\[
\begin{aligned}
\int_{\tO}\al\rhoe\,&\tG\ast(\al\rhoe)\,\dtx
=\int_{[1-\eta,1]}\frac{\al\calP(d\al)}{\calP([1-\eta,1])}\int_\Om\psie(x)\,dx
\int_{\tO}G(x,y)\beta\rhoe(y)\,\calP(d\beta)dy\\
=&\int_{[1-\eta,1]}\frac{\al\calP(d\al)}{\calP([1-\eta,1])}\int_\Om\psie(x)\,dx
\int_{[1-\eta,1]}\frac{\beta\calP(d\beta)}{\calP([1-\eta,1])}\int_\Om G(x,y)\psie(y)\,dy\\
=&\left(\frac{\int_{[1-\eta,1]}\al\,\calP(d\al)}{\calP([1-\eta,1])}\right)^2\int_\Om\psie\,G\ast\psie.
\end{aligned}
\]
We claim that
\begin{equation}
\label{eq:rhoe2}
\int_{\tO}\rhoe(\tx)\log\rhoe(\tx)\,\dtx=\int_\Om\psie(x)\log\psie(x)\,dx.
\end{equation}
Indeed, we have:
\[
\begin{aligned}
\int_{\tO}\rhoe(\tx)\log\rhoe(\tx)\,\dtx
=&\int_{[1-\eta,1]}\frac{\calP(d\al)}{\calP([1-\eta,1])}\int_\Om\psie\log(\chi_{[1-\eta,1]}(\al)\psie(x))\,dx\\
=&\int_{[1-\eta,1]}\frac{\calP(d\al)}{\calP([1-\eta,1])}\int_\Om\psie\log\psie(x)\,dx\\
=&\int_\Om\psie(x)\log\psie(x)\,dx.
\end{aligned}
\]
In view of \eqref{eq:rhoe1}--\eqref{eq:rhoe2} we may write
\[
\begin{aligned}
\calF(\rhoe)=\int_\Om\psie\log\psie
-\frac{1}{2}\left(\frac{\int_{[1-\eta,1]}\al\,\calP(d\al)}{\calP([1-\eta,1])}\right)^2\int_\Om\psie\,G\ast\psie-\la.
\end{aligned}
\]
In view of Lemma~\ref{lem:psieexp}, we deduce the expansion
\[
\calF(\rhoe)=\la\left\{1-\left(\frac{\int_{[1-\eta,1]}\al\,\calP(d\al)}{\calP([1-\eta,1])}\right)^2
\frac{\la}{8\pi+o(1)}\right\}\log\frac{1}{\eps^2}+O(1),
\]
as $\eps\to0^+$.
Since $\la>8\pi$, by taking $0<\eta\ll1,$ we may assume that
\[
\la>\left(\frac{\calP([1-\eta,1])}{\int_{[1-\eta,1]}\al\,\calP(d\al)}\right)^2\,8\pi.
\]
It follows that for some suitably small $\eps_0>0$ we have
\[
1-\left(\frac{\int_{[1-\eta,1]}\al\,\calP(d\al)}{\calP([1-\eta,1])}\right)^2
\frac{\la}{8\pi+o(1)}<0
\]
for all $0<\eps<\eps_0$, and the desired asymptotic behavior~\eqref{eq:Finfty} follows.
\par
The proof of Theorem~\ref{thm:Beckner} is now complete.
\end{proof}
%%%%%%%%%%%%%%%%%%%%%%%%%%%%%%%%%%%%%%%%%%%%%%%%%%%%%%%%%%%%%%%%%%%%%%%%%%%%%%%%%%~~~~~
%%%%%%%%%%%%%%%%%%%%%%%%%%%%%%%%%%%%%%%%%%%%%%%%%%%%%%%%%%%%%%%%%%%%%%%%%%%%%%%%%%~~~~~
\section{Appendix: Proof of Lemma~\ref{lem:psieexp}}
\label{sec:appA}
%%%%%%%%%%%%%%%%%%%%%%%%%%%%%%%%%%%%%%%%%%%%%%%%%%%%%%%%%%%%%%%%%%%%%%%%%%%%%%%%%%~~~~~
%%%%%%%%%%%%%%%%%%%%%%%%%%%%%%%%%%%%%%%%%%%%%%%%%%%%%%%%%%%%%%%%%%%%%%%%%%%%%%%%%%~~~~~
We recall from Section~\ref{sec:Beckner} that
\[
\psie=\la\frac{e^{\Ue}}{\int_\Om e^{\Ue}},
\]
where $\Ue$ is the Liouville bubble defined in \eqref{def:bubble}.
In what follows we define:
\begin{equation}
\label{def:Omeps}
\Om_\eps:=\{y\in\rr^2:\ \eps y\in\Om\}.
\end{equation}
We compute:
\begin{equation}
\label{eq:psilogpsi}
\int_\Om\psie\log\psie=
\int_\Om\frac{\la}{\int_\Om e^{\Ue}}e^{\Ue}\log\left(\frac{\la}{\int_\Om e^{\Ue}}e^{\Ue}\right)
=\frac{\la}{\int_\Om e^{\Ue}}\int_\Om e^{\Ue}\Ue
+\la\log\left(\frac{\la}{\int_\Om e^{\Ue}}\right).
\end{equation}
Moreover,
\begin{equation}
\label{eq:psiGpsi}
\int_\Om\psie\,G\ast\psie
=\left(\frac{\la}{\int_\Om e^{\Ue}}\right)^2\int_\Om e^{\Ue}\,G\ast e^{\Ue}.
\end{equation}
\begin{lemma}
\label{lem:Ueint}
The following expansion holds, as $\eps\to0^+$:
\[
\int_{\Om}e^{\Ue}=8\pi+o(1).
\]
\end{lemma}
\begin{proof}
We have, recalling \eqref{def:Omeps}:
\[
\begin{aligned}
\int_\Om e^{\Ue}=\int_\Om\frac{8\eps^2}{(\eps^2+|x|^2)^2}\,dx
=8\int_{\Om_\eps}\frac{dy}{(1+|y|^2)^2}.
\end{aligned}
\]
Let $0<r_1<r_2$ be such that $B_{r_1}\subset\Om\subset B_{r_2}$.
We have, for $j=1,2$:
\[
\int_{B_{r_j/\eps}}\frac{dy}{(1+|y|^2)^2}=\pi\left(1-\frac{1}{1+(\frac{r_j}{\eps})^2}\right)
\]
so that
\[
8\pi\left(1-\frac{1}{1+(\frac{r_1}{\eps})^2}\right)\le\int_\Om e^{\Ue}\le8\pi\left(1-\frac{1}{1+(\frac{r_2}{\eps})^2}\right)
\]
and the claim follows.
\end{proof}
\begin{lemma}
\label{lem:eUeUe}
The following expansion holds, as $\eps\to0^+$:
\[
\int_\Om e^{\Ue}\Ue
=\log\left(\frac{1}{\eps^2}\right)\int_\Om e^{\Ue}+O(1),
\]
uniformly for $\eps\to0^+$.
\end{lemma}
\begin{proof}
We have:
\[
\int_\Om e^{\Ue}\Ue=\int_\Om e^{\Ue}\log\frac{8\eps^2}{(\eps^2+|x|^2)^2}
=\int_\Om e^{\Ue}\log\frac{1}{(\eps^2+|x|^2)^2}+\log(8\eps^2)\int_\Om e^{\Ue}.
\]
We simplify the first term:
\[
\begin{aligned}
\int_\Om e^{\Ue}\log\frac{1}{(\eps^2+|x|^2)^2}\,dx=&
\int_\Om e^{\Ue}\log\frac{1}{\eps^4(1+|\frac{x}{\eps}|^2)^2}\,dx\\
\stackrel{y=x/\eps}{=}&\log\frac{1}{\eps^4}\int_\Om e^{\Ue}+\int_{\Om/\eps}\frac{8}{(1+|y|^2)^2}\log\frac{1}{(1+|y|^2)^2}\,dy.
\end{aligned}
\]
The asserted expansion follows.
\end{proof}
We note that in view of \eqref{eq:Liouville} we may write
\begin{equation*}
G\ast e^{\Ue}=P\Ue,
\end{equation*}
where $P$ denotes the projection operator onto $H_0^1(\Om)$.
We recall that
\begin{equation}
\label{eq:projexp}
P\Ue=\Ue-\log(8\eps^2)+8\pi H(x,0)+O(\eps^2),
\end{equation}
where $H(x,y)$ is te Robin's function defined by
\[
G(x,y)=\frac{1}{2\pi}\log\frac{1}{|x-y|}+H(x,y),
\]
see, e.g., \cite{EGP}.
\begin{lemma}
\label{lem:eUeGeUe}
The following expansion holds:
\[
\int_\Om e^{\Ue}\,G\ast e^{\Ue}=\log\frac{1}{\eps^4}\int_\Om e^{\Ue}+O(1).
\]
\end{lemma}
\begin{proof}
Using \eqref{eq:projexp} we compute:
\[
\begin{aligned}
\int_\Om e^{\Ue}\,G\ast e^{\Ue}
=&\int_\Om e^{\Ue}P\Ue
=\int_\Om e^{\Ue}(\Ue-\log(8\eps^2)+O(1))\\
=&\log\left(\frac{1}{\eps^2}\right)\int_\Om e^{\Ue}-\log\eps^2\int_\Om e^{\Ue}+O(1).
\end{aligned}
\]
The claim follows.
\end{proof}
\begin{proof}[Proof of Lemma~\ref{lem:psieexp}]
Proof of (i).
In view of \eqref{eq:psilogpsi}, Lemma~\ref{lem:Ueint} and Lemma~\ref{lem:eUeUe}, we readily derive the desired expansion.
\par
Proof of (ii). In view of \eqref{eq:psiGpsi}, Lemma~\ref{lem:Ueint} and Lemma~\ref{lem:eUeGeUe}, we readily derive the desired expansion.
\end{proof}
%%%%%%%%%%%%%%%%%%%%%%%%%%%%%%%%%%%%%%%%%%%%%%%%%%%%%%%%%%%%%%%%%%%%%%%%%%%%%%%%%%~~~~~
%%%%%%%%%%%%%%%%%%%%%%%%%%%%%%%%%%%%%%%%%%%%%%%%%%%%%%%%%%%%%%%%%%%%%%%%%%%%%%%%%%
\section{Concluding remarks: comparison of two mean field equations}
\label{sec:appB}
%%%%%%%%%%%%%%%%%%%%%%%%%%%%%%%%%%%%%%%%%%%%%%%%%%%%%%%%%%%%%%%%%%%%%%%%%%%%%%%%%%
%%%%%%%%%%%%%%%%%%%%%%%%%%%%%%%%%%%%%%%%%%%%%%%%%%%%%%%%%%%%%%%%%%%%%%%%%%%%%%%%%%
We have rigorously established in Theorem~\ref{thm:var} that the functionals
\begin{equation*}
\begin{aligned}
\calL(\rho,v)=&\int_{\tO}\rho(\log\rho-1)\,\dtx+\frac{1}{2}\int_\Om|\nabla v|^2\,dx
-\int_{\tO}\al\rho v\,\dtx,\\
\calJ(v)=&\frac{1}{2}\int_\Om|\nabla v|^2\,dx-\la\log\left(\int_{\tO}e^{\al v}\,\dtx\right)+\la(\log\la-1),
\end{aligned}
\end{equation*}
where $\rho=\oplus\rhoa\in L\log L(\tO)$, $v\in H_0^1(\Om)$,
are related by the minimization property
\[
\calJ(v)=\min_{\widetilde\Gamma_\la}\calL(\cdot,v)
\qquad\hbox{for all\ }v\in H_0^1(\Om),
\]
where
\[
\begin{aligned}
\widetilde\Gamma_\la:=&\left\{\rho\in L\log L(\tO):\
\rho\ge0\quad \hbox{a.e.\ },\; \int_{\tO}\rho\,\dtx=\la\right\}.
\end{aligned}
\]
Moreover, Theorem~\ref{thm:var}--(iv) and Theorem~\ref{thm:Beckner} imply that the optimal value of $\la>0$
which ensures boundedness from below of $\calJ$ on $H_0^1(\Om)$ is given by
\begin{equation}
\label{eq:MTstoch}
\bar\la=8\pi.
\end{equation}
In view of the corresponding results for the case $\calP(d\al)=\de_1(d\al)$,
the value $\bar\la$ is expected to provide the critical total mass for the occurrence of chemotactic collapse
vs.\ the existence of global solutions for \eqref{eq:detsys}, as well
for the evolution problem
\begin{equation*}
\left\{
\begin{aligned}
&\frac{\pl v}{\pl t}=\Delta v+\la\int_{[-1,1]}\frac{\al e^{\al v}}{\int_{\tO}e^{\beta v}\,\dtx}\,\calP(d\al),
&&\mbox{in\ }\Om\times(0,T)\\
&v=0,
&&\mbox{on\ }\partial\Om\times(0,T)\\
&v(x,0)=v^0(x),
&&\mbox{in\ }\Omega.
\end{aligned}
\right.
\end{equation*}
See \cite{DolbeaultPerthame2004, KavallarisSuzuki, JaegerLuckhaus, GajewskiZacharias1998} and the references therein.
The critical value $\bar\la$ also plays a central role in establishing the existence of the corresponding steady states,
i.e., of solutions for the non-local semilinear elliptic problem
\begin{equation}
\label{eq:steadyneri}
\left\{
\begin{aligned}
-\Delta v=&\la\int_{[-1,1]}\frac{\al e^{\al v}}{\int_{\tO}e^{\beta v}\,\dtx}\,\calP(d\al),
&&\hbox{in\ }\Om\\
v=&0,&&\hbox{on\ }\pl\Om.
\end{aligned}
\right.
\end{equation}
See \cite{RiZe2016, DeMarchisRicciardi, RicciardiTakahashi, PistoiaRicciardi, CSLinreview}.
\par
It is interesting to compare the properties mentioned above with the corresponding results
recently obtained in \cite{RicciardiSuzuki}
for the \textit{same} Lyapunov functional~$\calL$ under a \textit{different} constraint for the conserved population mass.
Such conditions were originally motivated by the deterministic model for stationary turbulent flows with variable intensity
derived in \cite{SawadaSuzuki} along the approach introduced by Onsager, see \cite{Suzuki2008book} and the references therein.
\par
More precisely, for $\la>0$ we define the functional
\[
\calI_\la(v):=\frac{1}{2}\int_\Om|\nabla v|^2\,dx-\la\int_{[-1,1]}\log\left(\int_\Om e^{\al v}\,dx\right)\,\calP(d\al)+\la(\log\la-1).
\]
We recall from Section~\ref{sec:Beckner} that the set $\Gamma_\la$ is defined by
\[
\Gamma_\la:=\left\{\psi\in L\log L(\Om):\ \psi\ge0\ \hbox{a.e.\ in\ }\Om,\ \int_\Om\psi\,dx=\la\right\}
\]
and we define correspondingly
\[
\widetilde{\widetilde\Gamma}_\la:=\oplus_{\al\in[-1,1]}\Gamma_\la:=\left\{\oplus\rhoa:\ \rhoa\in\Gamma_\la\ \mbox{for all\ }\al\in[-1,1]\right\}.
\]
In words, $\widetilde{\widetilde\Gamma}_\la$ is the admissible set of population densities $\rhoa$, $\al\in I$,
\textit{all} of which have total mass $\la$, i.e., $\int_\Om\rhoa=\la$ for all $\al\in I$.
\par
The following duality property was rigorously established in \cite{RicciardiSuzuki}
in the same spirit as Theorem~\ref{thm:var}--(iv):
\[
\inf_{\widetilde{\widetilde\Gamma}_\la\times H_0^1(\Om)}\calL=\inf_{\widetilde{\widetilde\Gamma}_\la}\calF=\inf_{H_0^1(\Om)}\calI_\la.
\]
Moreover,
\[
\calI_\la(v)=\min_{\widetilde{\widetilde\Gamma}_\la}\calL(\cdot,v)
\ \hbox{for all\ }v\in H_0^1(\Om).
\]
This duality property, together with the logarithmic Hardy-Littlewood-Sobolev inequality established in \cite{ShafrirWolansky}, was
used to compute the optimal value of $\la$ which ensures boundedness from below of the functional $\mathcal I_\la$,
which is given by
\[
\bar{\bar\la}=\inf\left\{
\frac{8\pi\calP(K_\pm)}{[\int_{K_\pm}\al\,\calP(d\al)]^2}:\ K_\pm\subset I_\pm\cap\supp\calP
\right\},
\]
where we denote $I_+:=[0,1]$, $I_-:=[-1,0)$,
and where $K_\pm$ denotes a Borel subset of $I_\pm$.
In particular, $\bar{\bar\la}$ significantly depends on $\calP$.
The value $\bar{\bar\la}$ is expected to provide the critical mass for chemotactic collapse
vs.\ global existence of solutions for the evolution problem
\begin{equation}
\label{eq:detMFevol}
\left\{
\begin{aligned}
&\frac{\pl v}{\pl t}=\Delta v+\la\int_I\frac{\al e^{\al v}}{\int_{\Om}e^{\al v}\,dx}\,\calP(d\al)
&&\hbox{in\ }\Om\times(0,T)\\
&v(x,t)=0&&\hbox{on }\pl\Om\times(0,T)\\
&v(x,0)=v^0(x),&&\hbox{in\ }\Om,
\end{aligned}
\right.
\end{equation}
We note that \eqref{eq:detMFevol} is obtained from \eqref{eq:vintro} by assuming the
\lq\lq individual population mass conservation" constraint:
\begin{equation}
\label{eq:indmass}
\int_\Om\rhoa(x,t)\,dx=\la\qquad\hbox{for all }\al\in[-1,1].
\end{equation}
Condition~\eqref{eq:indmass} is natural when the population species do not evolve from one kind into another.
The value $\bar{\bar\la}$ also yields the first blow-up level for the corresponding steady state problem
\begin{equation}
\label{eq:detMFss}
\left\{
\begin{aligned}
-\Delta v=&\la\int_I\frac{\al e^{\al v}}{\int_{\Om}e^{\al v}\,dx}\,\calP(d\al),
&&\hbox{in\ }\Om\\
v=&0,&&\hbox{on }\pl\Om.\\
\end{aligned}
\right.
\end{equation}
Results for solutions to the stationary problem~\eqref{eq:detMFss} have been obtained in
\cite{ORS, JevnikarYang}. In particular, the special case $\calP(d\al)=(\de_1(d\al)+\de_{1/2}(d\al))/2$
was studied in \cite{JevnikarYang} in relation to the Tzitz\'eica equation in differential geometry.
\par
In short, the steady state analysis for the problems \eqref{eq:steadyneri} and \eqref{eq:detMFss}
shows that,
despite of their formal similarity and the fact that they are motivated by the same statistical mechanics problem,
the corresponding solution sets exhibit significantly different mathematical properties.
\par
By introducing the new multi-species chemotaxis system \eqref{eq:detsys},
we have shown that the stationary problems \eqref{eq:steadyneri} and \eqref{eq:detMFss}
may be \emph{both} viewed as steady states for the chemotaxis system \eqref{eq:detsys} in the fast population dynamics limit,
by imposing \emph{different} conserved population mass constraints given by \eqref{eq:averagemassintro} and \eqref{eq:indmass},
respectively; the former being natural in the situation where the populations $\rhoa$ are are produced by a cell differentiation
process, the latter in the situation where evolution from one species into another does not occur.
%%%%%%%%%%%%%%%%%%%%%%%%%%%%%%%%%%%%%%%%%%%%%%%%%%%%%%%%%%%%%%%%%%%%%%%%%%%%%%%%%%~~~~
%%%%%%%%%%%%%%%%%%%%%%%%%%%%%%%%%%%%%%%%%%%%%%%%%%%%%%%%%%%%%%%%%%%%%%%%%%%%%%%%%%
\section*{Acknlowledgements}
This research is partially supported by PRIN 2012 74FYK7\_005 and GNAMPA-INDAM 2015
\lq\lq Alcuni aspetti di equazioni ellittiche non-lineari".
The first author acknowledges warm hospitality at the Department of Mathematics and Applications
of Naples
Federico II University, where part of this work was carried out.
%%%%%%%%%%%%%%%%%%%%%%%%%%%%%%%%%%%%%%%%%%%%%%%%%%%%%%%%%%%%%%%%%%%%%%%%%%%%%%%%%%
%%%%%%%%%%%%%%%%%%%%%%%%%%%%%%%%%%%%%%%%%%%%%%%%%%%%%%%%%%%%%%%%%%%%%%%%%%%%%%%%%%
%%%%%%%%%%%%%%%%%%%%%%%%%%%%%%%%%%%%%%%%%%%%%%%%%%%%%%%%%%%%%%%%%%%%%%%%%%%%%%%%%%

\end{document}